\documentclass{amsart}
\usepackage{all2021,xypic}
\newcommand{\parm}{\mathbf{parameterise}}
\newcommand{\impl}{\mathbf{implicitise}}

\newcommand{\smear}{\mathbf{smear}}
\newcommand{\cert}{\mathbf{certify}}
\newcommand{\cV}{\mathcal{V}}
\newcommand{\Sh}{\mathrm{Sh}}
\renewcommand{\phi}{\varphi}
\renewcommand{\Vec}{\mathbf{Vec}}

\newcommand{\FI}{\mathbf{FI}}

\newcommand{\FIop}{\mathbf{FI^{op}}}
\newcommand{\Map}{\mathbf{Map}}

\begin{document}

\title{Implicitisation and parameterisation in polynomial functors}
\author{Andreas Blatter} 
\thanks{AB was supported by Swiss National Science
Foundation grant 200021\_191981}
\email{andreas.blatter@unibe.ch}
\address{Mathematical Institute, University of Bern,
Alpeneggstrasse 22, 3012 Bern, Switzerland}
\author{Jan Draisma}
\thanks{JD was partially supported by SNSF grant
200021\_191981 and Vici grant 639.033.514 from
the Netherlands Organisation for Scientific Research}
\email{jan.draisma@unibe.ch}
\address{Mathematical Institute, University of Bern,
Sidlerstrasse 5, 3012 Bern, Switzerland; and Department of
Mathematics and Computer Science, Eindhoven University of
Technology, P.O.~Box 513, 5600MB Eindhoven, The Netherlands}
\author{Emanuele Ventura}
\email{emanuele.ventura@polito.it, emanueleventura.sw@gmail.com}
\address{Politecnico di Torino, 
Dipartimento di Scienze Matematiche ``G.L. Lagrange'', Corso Duca degli Abruzzi 24
10129 Torino, Italy}

\maketitle

\begin{abstract}
In earlier work, the second author showed that a closed subset
of a polynomial functor can always be defined by finitely many
polynomial equations. In follow-up work on $\GL_\infty$-varieties,
Bik-Draisma-Eggermont-Snowden showed, among other things, that in
characteristic zero every such closed subset is the image of a morphism
whose domain is the product of a finite-dimensional affine variety and
a polynomial functor. In this paper, we show that both results can be
made algorithmic: there exists an algorithm $\impl$ that takes as input
a morphism into a polynomial functor and outputs finitely many equations
defining the closure of the image; and an algorithm $\parm$ that takes as
input a finite set of equations defining a closed subset of a polynomial
functor and outputs a morphism whose image is that closed subset.
\end{abstract}

\maketitle

\section{Introduction}

\subsection{Implicitisation}
An important theme in computational algebraic geometry is {\em
implicitisation}: given a list $(\phi_1,\ldots,\phi_n)$ of polynomials in
$K[x_1,\ldots,x_m]$, which represent a polynomial map $\phi$ from
the $m$-dimensional affine space $\AA^m$ to $\AA^n$ over the field $K$, the challenge is to
compute equations for the Zariski closure
$\overline{\im(\phi)}$ of the image of $\phi$. This
challenge is solved, at least theoretically, by {\em elimination} using
Buchberger's algorithm; e.g., see~\cite[\S3.3]{CLS15}.

\subsection{Implicitisation in families}
But now consider a scenario where one is given not a single polynomial
map, but rather a family $\phi_i:\AA^{m_i} \to \AA^{n_i}$ of polynomial
maps depending on a discrete parameter $i$ that takes infinitely
many values. If the ambient spaces $\AA^{m_i}$ and $\AA^{n_i}$ and
the polynomial map $\phi_i$ vary favourably with $i$, it is sometimes
possible to find finitely many equations that capture the image closures
of all $\phi_i$ at once. 

\subsection{Implicitisation over categories}
\label{ssec:Category}
To make the phrase ``vary favourably'' concrete, assume that 
$i$ ranges through the objects of a category $C$ and $\AA(i)$ is a
finite-dimensional affine space varying functorially with
$i$. This means that 
for each $\pi \in \Hom_C(i,j)$ we have a linear map $\AA(\pi):\AA(i)
\to \AA(j)$ such that $\AA(\id_i)=\id_{\AA(i)}$ and 
$\AA(\sigma \circ \pi)=\AA(\sigma) \circ \AA(\pi)$ for
$\sigma \in \Hom_C(j,k)$.
Suppose that $\AA'(i)$ is another affine space depending
functorially on $i \in C$, and finally that, for any $i \in C$, we have
a polynomial map $\phi_i:\AA(i) \to \AA'(i)$ such that the following
diagram commutes for every $\pi \in \Hom_C(i,j)$:
\[ 
\xymatrix{\AA(i) \ar[r]^{\phi_i} \ar[d]_{\AA(\pi)} & \AA'(i)
\ar[d]^{\AA'(\pi)} \\
\AA(j) \ar[r]_{\phi_j} & \AA'(j).}
\]
In this case, if $f$ is a polynomial equation for the image
closure $\overline{\im(\phi_j)}$ of
$\phi_j$, then $f \circ \AA'(\pi)$ is a polynomial equation
for $\overline{\im(\phi_i)}$---and, since we assumed that $\AA'(\pi)$
is linear, of the same degree. We then ask:
\begin{enumerate}
\item Do there exist finitely many $j \in C$ such that the equations for
those $\overline{\im(\phi_j)}$, by pulling back along the linear
maps $A'(\pi)$ for all relevant $\pi$, define
$\overline{\im(\phi_i)}$ for all $i \in C$?
\item If so, does there exist an {\em algorithm} for computing these
finitely many $j$?
\end{enumerate}

A well-known case where the answer to the two questions above is ``yes''
is that where $C$ is the opposite category $\FIop$ of the category $\FI$
of finite sets with injections, $\AA(i)=(\AA^n)^i$, $\AA'(i)=(\AA^{n'})^i$
for some fixed $n$ and $n'$, and the maps $\AA(i) \to \AA(j),
\AA'(i) \to \AA'(j)$ corresponding to an injection $j \to i$ are the
canonical projections. In that case, it is known that the kernel of the
$\FI$-homomorphism $\phi^*$ dual to the $\FIop$-polynomial map $\phi$ is
finitely generated and can be computed using a version of Buchberger's
algorithm, and this has been applied to problems in algebraic
statistics \cite{Cohen87,Hillar09,Hillar13b,Hillar18,Brouwer09e}. The
setting discussed in the current paper is of a very different flavour
in that it involves continuous symmetries rather than discrete
symmetries, and it is also significantly more complicated. One cause for trouble is that in the
setting below, we do not actually know whether the kernel of the relevant
algebra homomorphism is finitely generated, and so we have to settle
for finding set-theoretic equations for the $\overline{\im(\phi_i)}$.

\subsection{Implicitisation in polynomial functors}
In the case where $C$ is the category $\Vec$ of finite-dimensional
$K$-vector spaces, and both $\AA$ and $\AA'$ are polynomial functors
$\Vec \to \Vec$ (see below for definitions), the second author
established a positive answer to the first question above
\cite{Draisma17}; see
Theorem~\ref{thm:Finiteness} below.

\begin{ex} \label{ex:Cubics}
Here are two instances of this setting:
\begin{enumerate}
\item $\AA(V)=V^k$ for some positive integer $k$, $\AA'(V)=S^3 V$, the
third symmetric power of $V$, and 
\[ \phi_V(v_1,\ldots,v_k):=v_1^3 + \cdots + v_k^3. \]
The image of $\phi_V$ is the set of cubic homogeneous polynomials in $\dim(V)$
variables of {\em Waring rank $\leq k$}, and its closure is the variety
of polynomials of {\em border Waring rank $\leq k$}. 

\item $\AA(V)=(S^2 V)^k \times V^k$, $\AA'(V)=S^3 V$, and
\[ \phi_V(q_1,\ldots,q_k,v_1,\ldots,v_k):=q_1 v_1 + \cdots +
q_k v_k. \]
In this case, the image of $\phi_V$ is closed
\cite{Derksen17} and consists of
the cubics of {\em q-rank $\leq k$}. 
\end{enumerate}
In the first case, it is well-known that generators of the
vanishing ideal of
$\overline{\im(\phi_U)}$ for $U$ of dimension $k+1$ pull
back to generators of the ideal of $\overline{\im(\phi_V)}$
for all $V$; this is called {\em symmetric inheritance} in
\cite{Landsberg10}. Much more is known for small values of
$k$; in particular, for $k \leq 2$, the ideal is generated
by $3 \times 3$-minors of catalecticant matrices
\cite{Raicu10b}. 

In the second case, we do not know whether the {\em ideal} is finitely
generated in this sense, but \cite{Derksen17} assures the existence of a
number $\ell=\ell(k)$ such that equations for $\overline{\im(\phi_U)}$
with $U$ of dimension $\ell$ pull back to equations that define
$\overline{\im(\phi_V)}$ for all $V$ {\em set-theoretically}.  No explicit
value of $\ell(k)$ is known, even for $k=2$.
\end{ex}

In many respects, the q-rank example is much more difficult than the
Waring rank example. The main reason for this is that the source $\AA$
is a polynomial functor of degree $2$, whereas $\AA$ has degree $1$ in
the first case. Nevertheless, Theorem~\ref{thm:Finiteness} implies, for
general polynomial functors $\AA,\AA'$ and morphisms $\phi$, the existence
of a $U$ such that in $\AA'(U)$ one sees enough
set-theoretic equations to define $\overline{\im(\phi_V)}$
for all $V$.

\subsection{Main result}

The main result of this paper makes
Theorem~\ref{thm:Finiteness} effective.

\begin{thm}[Main Theorem; see Theorem~\ref{thm:impl} for a
more precise statement]
There exists an algorithm $\impl$ that, on input polynomial functors
$\AA,\AA'$ and a morphism $\phi:\AA \to \AA'$, computes
a $U \in \Vec$ such that the equations for
$\overline{\im(\phi_U)}$
pull back to set-theoretic defining equations for 
$\overline{\im(\phi_V)}$ for all $V$.
\end{thm}

This means, in particular, that the number $\ell(k)$ for the cubics of
q-rank at most any fixed $k$ can in principle be computed.

\subsection{Structure of the algorithm}
\label{ssec:Structure}

The structure of $\impl$ is as follows. For $U=K^0,K^1,K^2,\ldots$
one computes equations for $\overline{\im(\phi_U)}$ using Buchberger's
algorithm. By Theorem~\ref{thm:Finiteness} we know that at some point
these equations, via pull-back, define $\overline{\im(\phi_V)}$ for all
$V$. However, Theorem~\ref{thm:Finiteness} does not give a criterion
for when we may stop---this is different from the algorithm in the
$\FI$-setting: there, if one has not seen new equations arise between
the the finite set $\{1,\ldots,n\}$ and the finite set
$\{1,\ldots,2n\}$, one is guaranteed to have found a generating set of
equations \cite{Brouwer09e}.

To obtain a stopping criterion in the polynomial functor setting, we
derive an algorithm $\parm$ that, on input the polynomial equations
on $\AA'(U)$ found so far, computes a parameterisation $\psi$ of the
closed subset of $\AA'$ defined by those equations. If we can check that
$\im(\psi) \subseteq \overline{\im(\phi)}$, then we are done.

That such a parameterisation $\psi$ exists is guaranteed by the {\em
unirationality theorem} from \cite{Bik21}. The algorithm $\parm$ makes this
theorem effective. 

Finally, to check that $\im(\psi)$ lies in the closure of $\im(\phi)$,
we pass to infinite dimensions and use a result from \cite{Bik22}
that says that this happens if and only if a suitably generic point
of $\im(\psi_\infty)$ can be reached as the limit of a curve in
$\im(\phi_\infty)$. In the current paper we show that this curve, which
lives in infinite-dimensional space, can be represented in finite terms,
and searched for, on a computer.

If such a curve does not exist, i.e., if $\im(\psi)$ is not contained
in $\overline{\im(\phi)}$, then this is because the current space
$U$ is too small. In this case, the search for a curve does not
terminate. So for our algorithm $\impl$ to 
terminate, it is imperative to run the search for witness curves in
parallel to the search for equations: in each step, $U$ is increased
with one dimension, new equations are computed, and a new curve search
is started. We model this behaviour by running the algorithm on countably
many parallel processors. Of course, standard results in the
theory of computation imply that this algorithm can then
also run on an ordinary Turing machine. 

\subsection{Outlook}

The algorithms $\impl$ and $\parm$ are very much theoretical algorithms,
and---apart from recent research on varieties over categories---they
rely on many powerful results from classical, finite-dimensional,
computational algebra: Buchberger's algorithm, of course, but also
algorithms for computing the radical of an ideal and for computing
the minimal primes containing a radical ideal.  Also, a fair amount of
representation theory of the general linear group goes into the algorithm.

At present, a general implementation of $\impl$ and $\parm$ seems entirely
out of reach. However, we believe that our computational approach
to implicitisation in polynomial functors will in the future serve as
a guide towards finding equations in concrete settings, e.g.~for the
variety of cubics of q-rank $\leq 2$ from
Example~\ref{ex:Cubics}.

\subsection{Organisation of the paper}

In Section~\ref{sec:Preliminaries}, we collect material that we will
need in the rest of the paper: some assumptions on the ground field,
computer representations of
finite-dimensional affine varieties, polynomial functors and their closed
subsets, and morphisms between these.

In Section~\ref{sec:Parameterisation}, we derive the algorithm $\parm$;
see Theorem~\ref{thm:parm}. In Section~\ref{sec:Implicitisation}
we derive the algorithm $\impl$ and prove our Main Theorem; see
Theorem~\ref{thm:impl} for a more precise statement.

Finally, $\impl$ depends on a procedure called $\cert$ that aims at certifying the
existence of a curve as discussed in \S\ref{ssec:Structure}. This procedure requires tools
from infinite-dimensional algebraic geometry, both from a forthcoming
paper \cite{Bik22} and from more classical sources. In particular, it is
Greenberg's approximation theorem
(Theorem~\ref{thm:Greenberg}) from \cite{Greenberg66} and
its effective version in \cite{Rond18} which
allows us to represent this curve in finite terms.

\subsection*{Acknowledgments}

The second author thanks Arthur Bik, Rob Eggermont, and Andrew Snowden
for numerous discussions on polynomial functors, especially on the topic
of systems of variables from \S\ref{ssec:Systems} and on formal curves
approximating a point in a closure, Theorem~\ref{thm:Limit}.

\section{Preliminaries}\label{sec:Preliminaries}

\subsection{The ground field}

Let $K$ be a field of characteristic zero in which we can do computations
on a computer. More precisely, we want $K$ to be a computable field,
and we further require that there exists an algorithm for factoring polynomials
in $K[X]$. The class of such fields is quite large; it includes $\QQ$
and its finitely generated extensions \cite[Appendix B]{GP08}.

\subsection{Finite-dimensional affine varieties}

A finite-dimensional affine variety over $K$ is represented by a finite
list of generators of a radical ideal in some polynomial ring over $K$
with finitely many variables. A morphism between affine varieties over
$K$ is represented by a finite list of polynomials.  Our algorithms will
intensively use existing algorithms for dealing with affine varieties.
Good general references are \cite{CLS15, CLS05}.  In particular, we will need
algorithms for computing the radical of an ideal; see e.g. \cite{KrLo91}. 

If $R$ is a $K$-algebra and $h$ is an element of $R$, then we write
$R[1/h]$ for the localisation. Similarly, if $B$ is a finite-dimensional
affine variety and $h$ an element of $K[B]$, then we write $B[1/h]$
for the basic open subset defined by $h$, i.e., the affine variety with
coordinate ring $K[B][1/h]$.

\subsection{Polynomial functors}

Let $\Vec$ be the category of finite-dimensional $K$-vector spaces. We
write $\Hom(U,V)$ and $\End(U)$ for the spaces of $K$-linear maps $U \to V$
and $U \to U$, respectively.

\begin{de}
A {\em polynomial functor of degree at most $d$} over $K$ is a covariant
functor $P:\Vec \to \Vec$ such that for any $U,V \in \Vec$ the
map $P:\Hom (U,V) \to \Hom_{P(U),P(V)}$ is polynomial of degree at
most $d$. Polynomial functors of degree at most $d$ form an abelian
category in which a homomorphism $P \to Q$ is a natural transformation,
i.e., given by a linear map $\psi_U:P(U) \to Q(U)$ for each $U \in \Vec$, 
such that for all $\phi \in \Hom(U,V)$ we have $Q(\phi) \circ \psi_U =
\psi_V \circ P(\phi)$.
\end{de}

When we say polynomial functor, we will always mean a polynomial functor
of degree at most some integer. {\em A priori}, a polynomial functor
seems to be given by an infinite amount of data. But up to isomorphism,
the following lemma due to Friedlander-Suslin \cite[Lemma 3.4]{Friedlander97} implies that it is given
by a finite amount of data only.

\begin{lm}[\cite{Friedlander97}] \label{lm:FS}
The map from polynomial functors of degree at most $d$ to
$\GL_d$-representations that assigns to $P$ the vector space $P(K^d)$
with the algebraic group homomorphism  $\GL_d \to \GL(P(K^d)), \phi
\mapsto P(\phi)$ is an equivalence of abelian categories from
polynomial functors of degree at most $d$ to polynomial
$\GL_d$-re\-pre\-sen\-tations of degree at most $d$. 
\end{lm}

Here, a homomorphism $\GL_n \to \GL(V)$ of algebraic groups is called
polynomial of degree at most $d$ if it extends to a polynomial map
$\End(K^n) \to \End(V)$ of degree at most $d$.

The Friedlander-Suslin lemma is relevant to us for two reasons. First,
since polynomial representations of $\GL_d$ are completely reducible and
the irreducible ones are completely classified by combinatorial data,
the same holds for polynomial functors. The consequence is that each
polynomial functor of degree at most $d$ is a direct sum of {\em Schur
functors}: functors of the form $S_\lambda: V \mapsto \Hom(U_\lambda,V^{\otimes
e})$, where $e \leq d$, $\lambda$ is a partition of $e$, and $U_\lambda$
is the corresponding irreducible representation of the symmetric group
on $e$ letters. So we may represent a polynomial functor by a finite
tuple of partitions. Second, the proof of the Friedlander-Suslin lemma
is completely constructive: it can be transformed into 
algorithms that compute, from a polynomial representation $\rho$ of $\GL_d$ of degree
$\leq d$ corresponding to a polynomial functor $P$:
\begin{itemize}
\item (a basis for) $P(V)$ on input $V$; 
\item (a matrix for) $P(\phi):P(U) \to P(V)$ on input a linear map 
$\phi:U \to V$; and 
\item (a matrix for) $\psi_V:P(V) \to Q(V)$ on input a second polynomial 
$\GL_d$-representation of degree $\leq d$ representing the polynomial
functor $Q$, as well as homomorphism of
$\GL_d$-representations. 
\end{itemize}
We will not make these algorithms explicit here.

Each polynomial functor $P$ of degree at most $d$ is a direct sum of {\em
homogeneous} polynomial functors: $P=P_0 \oplus \cdots \oplus P_d$, where
\[ P_e(V):=\{p \in P(V) \mid P(t \id_V)p=t^e p\}.\] 
In particular, $P_0$
is a degree-$0$ polynomial functor, which assigns a fixed vector space,
also denoted $P_0$, to every $V$ and the identity $\id_{P_0}$ to every
linear map $\phi \in \Hom(U,V)$. We call $P$ {\em pure} if $P_0 = 0$,
and we call $P_1 \oplus \cdots \oplus P_d$ the {\em pure part} of $P$.

\subsection{Closed subsets of polynomial functors}

\begin{de}
A {\em closed subset} $X$ of a polynomial functor $P$ over $K$ is the data
of a closed subvariety $X(V)$ of $\ol{K} \otimes P(V)$ defined over $K$
such that for each $\phi \in \Hom(U,V)$ the linear map $1 \otimes P(\phi)$
maps $X(U)$ into $X(V)$.
\end{de}

We take points with coordinates in $\ol{K}$ because $K$ might not be
large enough to see all points. On the other hand, our algorithm will
always work over $K$ itself, and all varieties will be defined over
$K$. In fact, we shall usually drop the $\ol{K}$ from the notation and
just write {\em closed subvariety $X(V)$ of $P(V)$} and $P(\phi):X(U)
\to X(V)$ in the above setting. This is not different from the classical
setting of implicitisation, where one computes the equations
of the Zariski closure in $\ol{K}^n$ of the image of a polynomial map
$\ol{K}^m \to \ol{K}^n$ which is defined over $K$.

For a closed subset $X$ of a polynomial functor $P$ and a linear map $\phi \in
\Hom(U,V)$, we will write $X(\phi):X(U) \to X(V)$ for the restriction
of $P(\phi)$ to $X(U)$. Note that, in particular, the map $\GL(V) \times
X(V) \to X(V), (g,x) \mapsto X(g)(x)$ defines an algebraic action of $\GL(V)$
on $X(V)$, for each $V \in \Vec$. So this paper is concerned with 
(typically) highly symmetric varieties, related by linear maps coming
from linear maps between distinct vector spaces. 

If $X$ is a closed subset of $P$, then the variety $B:=X(0)$ is a closed
subvariety of $P(0)=P_0$. For each $V \in \Vec$, the zero map $0_{V,0}:
V \to 0$ maps $X(V)$ to $X(0)$, and indeed {\em onto} $X(0)$, because
for any $p_0 \in X(0) \subseteq P(0)$ we have
\[ p_0=P(0_{V,0} \circ 0_{0,V})p_0=P(0_{V,0})(P(0_{0,V})p_0) \in
P(0_{V,0})(X(V)) \]
where we have used that $P$ is a functor and $X$ is a closed subset of
$P$. Hence for each $V$, $X(V)$ is a closed subset of $B \times P'(V)$,
where $P'$ is the pure part of $P$, such that $X(V)$ maps surjectively
to $B$. We will also say that {\em $X$ is a closed subset of
$B \times P'$}.

The following theorem with its corollary implies that a closed subset $X$ 
of a polynomial functor $P$ can be represented on a computer.

\begin{thm}[\cite{Draisma17}] \label{thm:Finiteness}
Let $P$ be a polynomial functor and let $X_1 \supseteq X_2 \supseteq \ldots$ be a descending chain of closed subsets of $P$. 
Then this chain stabilizes, i.e., there exists $n_0$ such that $X_{n_0} = X_{n_0 + 1} 
= \ldots$
\end{thm}

\begin{cor} \label{cor:Finiteness}
For each closed subset $X \subseteq P$ there exists a vector space $U
\in \Vec$ with the property that for all $V \in \Vec$ we have
\[ X(V)=\bigcap_{\phi \in \Hom(V,U)} P(\phi)^{-1}(X(U)). \]
\end{cor}

\begin{proof}
For every $n\in \NN$ consider the closed subset $X_n$ in $P$
defined by 
\[ X_n(V):=\bigcap_{\phi \in \Hom(V,K^n)} P(\phi)^{-1}(X(K^n)). \]
Using that $P$ is a polynomial functor and $X$ is a closed subset, we can see that
for all $n$, $X_n(K^n) = X(K^n)$ and that $X_1 \supseteq X_2 \supseteq \ldots$ is
a descending chain. By Theorem~\ref{thm:Finiteness} this chain stabilizes at, say,
$X_{n_0}$, and hence for $n\geq n_0$, $X(K^n) = X_n(K^n) = X_{n_0}(K^n)$.
Using that $X$ is a closed subset, we get the same for $n<n_0$. Hence, 
$X=X_{n_0}$, and the corollary holds for $U=K^{n_0}$.
\end{proof}

Consequently, if $f_1,\ldots,f_k \in K[P(U)]$ are defining
equations for $X(U)$, then the pull-backs of the $f_i$ along all
linear maps $P(\phi):P(V) \to P(U)$ cut out $X(V)$. We then write
$X=\cV_P(f_1,\ldots,f_k)$. The tuple consisting of $P$, $U$, and
$(f_1,\ldots,f_k)$ together form a computer representation of
$X$. Slightly more generally, we will often represent $X$ as follows. 

\begin{de}
Let $B$ be a finite-dimensional affine variety, $P$ a pure polynomial
functor, $U$ a finite-dimensional vector space, and 
$f_1,\ldots,f_k \in K[B \times P(U)]$. The tuple
$(B,P,U,f_1,\ldots,f_k)$ is called the {\em implicit representation}
of the closed subset $X$ of $B \times P$ defined by 
\[ X(V)=\{(b,p') \mid \forall \phi \in \Hom(V,U) \ \forall
i=1,\ldots,r: f_i(b,P(\phi)p')=0\}.\qedhere \]
\end{de}

Again by Corollary~\ref{cor:Finiteness}, every closed subset of $B \times
P$ admits an implicit representation. Here we allow that some of the $f_i$
are nonzero elements of $K[B]$, so that $X$ does not map surjectively
to $B$ but rather onto a closed subvariety of $B$.

\subsection{Irreducibility}

\begin{de}
Let $B$ be a finite-dimensional affine variety, $Q$ a pure polynomial
functor, and $X$ a closed subset of $B \times Q$.  Then $X$ is called
irreducible if $X \neq \emptyset$ and whenever $X=X_1 \cup X_2$ where
$X_1,X_2$ are closed subsets of $B \times Q$, we have $X=X_1$ or $X=X_2$.
\end{de}

A straightforward check shows that $X$ is irreducible if and only if
$X(V)$ is irreducible for each $V \in \Vec$; in particular, $B
\times Q$ is irreducible if and only if $B$ is irreducible. Note that
since $K$ may not be algebraically closed, irreducibility in our sense
may not imply irreducibility over $\ol{K}$.

\subsection{Gradings and ideals} \label{ssec:Gradings}

For any polynomial functor $P=P_0 \oplus \cdots \oplus P_d$, and any
$V \in \Vec$, the coordinate ring $K[P(V)] \cong \bigotimes_{e=0}^d
K[P_e(V)]$ is a graded polynomial ring in which the coordinates on $K[P_e(V)]$
are given the degree $e$. A polynomial $f$ is homogeneous of degree $n$
with respect to this grading if and only if $f(P(t \cdot \id_V) p)=t^n f(p)$
for all $p \in P(V)$. If $X$ is a closed subset of $P$, then $X(V)$ is
preserved under $P(t \cdot \id_V)$ for all $t$, and hence the ideal of $X(V)$
is homogeneous.

Similarly, for $B$ a finite-dimensional affine variety and $P$ a pure
polynomial functor, $K[B \times P(V)]$ has a standard grading in which
the elements of $K[B]$ have degree $0$, and the ideal of any closed
subset $X$ of $B \times P$ is homogeneous.

We find that the coordinate rings $K[X(V)]$, for $X$ a closed subset
of a polynomial functor $P$ of degree at most $d$ (or for $X$ a closed
subset of $B \times P$ with $P$ a pure polynomial functor of degree
$d$), have standard gradings and are generated in degree at most $d$. A
straightforward computation shows that for each $\phi \in \Hom(U,V)$,
the pullback $X(\phi)^*:K[X(V)] \to K[X(U)]$ is a graded $K$-algebra
homomorphism.

\subsection{Morphisms}

\begin{de}
Let $X,Y$ be closed subsets of polynomial functors. Then a morphism
$\alpha:X \to Y$ is given by a morphism $\alpha_V:X(V) \to Y(V)$ of
affine varieties over $K$ for each
$V \in \Vec$ such that for all $\phi \in \Hom(U,V)$ we have $Y(\phi)
\circ \alpha_U=\alpha_V \circ X(\phi)$.
\end{de}

By taking for $\phi$ scalar multiples of the identity, one finds that
the pull-back $\alpha_V^*$ is a graded $K$-algebra homomorphism.

Lemma \ref{lm:Unique} ensures that we can represent morphisms on a
computer. First, the following generalises a well-known property of
finite-dimensional affine varieties (see \cite[Proposition 1.3.22]{Bik2020}).

\begin{lm} \label{lm:Extends}
If $X$ is a closed subset of a polynomial functor $P$ and $Y$ is a
closed subset of a polynomial functor $Q$, then any morphism $X \to Y$
extends to a morphism $P \to Q$.
\end{lm}

\begin{lm} \label{lm:Unique}
Let $X,Y$ be closed subsets of polynomial functors of degree at
most $d$. Then a morphism $\alpha:X \to Y$ is uniquely determined
by $\alpha_{K^d}:X(K^d) \to Y(K^d)$, and this unique determination is
algorithmic in the sense that if $\alpha_{K^d}$ is known, then $\alpha_V$
can be computed for any $V \in \Vec$.
\end{lm}

\begin{proof}
By Lemma~\ref{lm:Extends}, $\alpha$ extends to a morphism $\beta:P
\to Q$, where $X$,$Y$ are closed subsets in the degree-$\leq d$
polynomial functors $P,Q$. Now for each $e=0,\ldots,d$, the restriction
of $\beta_V^*:Q_e(V)^* \to K[P(V)]_{e}$ defines, as $V$ varies,
a homomorphism from the polynomial functor $V^* \mapsto Q_e(V)^*$
to the polynomial functor $V^* \mapsto K[P(V)]_e$, both of degree $e$
(except that $K[P(V)]_e$ is not finite-dimensional if $P_0 \neq 0$, but
we may replace it by its image). Then we apply the Friedlander-Suslin
lemma to conclude that this homomorphism is uniquely determined by its
evaluation at $V=K^d$. The algorithmicity follows from the algorithmicity
of the Friedlander-Suslin lemma.
\end{proof}

\begin{re} \label{re:Map}
A special instance of the lemma is when $X=P$ and $Y=Q$, with 
$P,Q$ pure polynomial functors of degree $\leq d$. In this case,
the space of morphisms $\Map(P,Q)$ is a finite-dimensional vector space
over $K$, namely, the direct sum for $e=1,\ldots,d$ of the space of
$\GL_d$-equivariant linear maps $Q_e(V)^* \to K[P(V)]_e$,
where $V=K^d$. This space will be important to us towards the end of
the paper.
\end{re}

\begin{re} \label{re:Split}
We will often consider morphisms $\alpha: A \times P \to B \times Q$,
where $A,B$ are finite-dimensional varieties over $K$ and $P,Q$ are
pure polynomial functors. Such a morphism decomposes into a morphism
$\alpha^{(0)}:A \to B$ and a morphism $\alpha^{(1)}:A \times P \to Q$;
here we use that $\alpha_V^*$ preserves the degree and that the coordinates on 
$B$ have degree zero (see \S\ref{ssec:Gradings}), and hence
their images cannot involve the positive-degree coordinates on $P$.  If $A$ is
irreducible, then $\alpha^{(1)}$ can be thought of as a $K(A)$-valued point of the
finite-dimensional affine space $\Map(P,Q)$; this will be a useful point
of view later.
\end{re}

\subsection{Shifting}

\begin{de}
Any fixed $U \in \Vec$ defines a covariant polynomial functor $\Sh_U:\Vec
\to \Vec$ of degree $1$ by $\Sh_U(V)=U\oplus V$ and, for $\phi:V \to
W$, $\Sh_U(\phi)=\id_U \oplus \phi$. If $P$ is a polynomial functor of
degree $d$, then $\Sh_U P:=P \circ \Sh_U$ is also a polynomial functor
of degree $d$, called the {\em shift over $U$} of $P$.
\end{de}
The top-degree parts of $\Sh_U P$ and $P$ are canonically
isomorphic \cite[Lemma 14]{Draisma17}. If $X$ is a closed subset of $P$, then $\Sh_U X:=X \circ \Sh_U$ is a closed
subset of $\Sh_U P$. Note that $(\Sh_U X)(0)=X(U)$, so shifting has
the effect of making the finite-dimensional base variety larger. More
precisely, if $X$ is a closed subset of $B \times P$ with $P$ a pure
polynomial functor, then $\Sh_U X$ is a closed subset of $(B \times P(U))
\times P'$ where $P'$ is the pure part of $\Sh_U P$.

\subsection{An order on polynomial functors} \label{ssec:Order}

\begin{de}
We call a polynomial functor $Q$ smaller than a polynomial functor $P$
if the two are not isomorphic and for the largest $e$ such that $Q_e$
is not isomorphic to $P_e$, the former is a quotient of the latter.
\end{de}

Using the Friedlander-Suslin lemma, one can show that this is a
well-founded order on polynomial functors \cite[Lemma 12]{Draisma17}.

\subsection{Summary}
We have now indicated computer representations for all the mathematical
objects that we will need below: finite-dimensional affine varieties,
polynomial functors, closed subsets of the latter and
morphisms between these. 
Furthermore, we have introduced two tools in the design
and analysis of our algorithms: shifting and a well-founded
order on polynomial functors.  

\section{Parameterisation}\label{sec:Parameterisation}

\subsection{The result}

The goal of this section is to prove the following theorem.

\begin{thm} \label{thm:parm}
There exists an algorithm $\parm$ that, on input a finite-dimensional
affine variety $B$, a pure polynomial functor $Q$, a finite-dimensional
vector space $U$ over $K$, and elements $f_i \in K[B \times Q(U)]$ for
$i=1,\ldots,k$, computes $(A;P;\beta)$ where $A$ is a finite-dimensional
affine variety, $P$ is a pure polynomial functor, and $\beta$ is a
morphism $A \times P \to B \times Q$, defined over $K$, such that for
each finite-dimensional $K$-vector space $V$ we have
\begin{align*} &\beta(A(\ol{K}) \times \ol{K} \otimes P(V))
\\
& = \{ (b,q) \in B(\ol{K}) \times \ol{K} \otimes Q(V) \mid
\forall i\  \forall \phi \in
\Hom_K(V,U):\ f_i(b, 1 \otimes Q(\phi)q) =0 \}.
\end{align*}
\end{thm}

Here, $Z(\ol{K})$ means the $\ol{K}$-points of a finite-dimensional affine
variety $Z$ defined over $K$, i.e., the set of $K$-algebra homomorphisms
$K[Z] \to \ol{K}$. If $Z$ is just a vector space over $K$, then this
means $\ol{K} \otimes Z$.  In what follows, to simplify notation, we
will supress $\ol{K}$ and write $(b,q) \in B \times Q(V)$ when we really
mean $(b,q) \in B(\ol{K}) \times Q(V)(\ol{K})$, and for a $\phi \in
\Hom_K(V,W)$ we will write $(b,Q(\phi)q) \in B \times Q(W)$ instead of
$(b,(1 \otimes Q(\phi))q) \in B(\ol{K}) \times Q(W)(\ol{K})$. Similarly,
we will write $X(V)$ even when we mean $X(V)(\ol{K})$.

We write $X:=\cV_{B \times Q}(f_1,\ldots,f_k)$, so that $X(V)$
is the set on the right-hand side above. This is the closed subset of
$B \times Q$ that we want to parameterise.

\begin{re}
We allow $A$ to be reducible, even when $B$ is irreducible. The fact
that $\alpha:A \times P \to B \times Q$ as in the theorem exists is
\cite[Theorem 4.2.5]{Bik2020}; see also \cite[Proposition 5.6]{Bik21}.  If $B$ is indeed irreducible, then for
some irreducible component $A'$ of $A$ the restriction of $\alpha$ to $A'
\times P$ is dominant into $\cV_{B \times Q}(f_1,\ldots,f_k)$.
\end{re}

\subsection{Smearing out equations}

Before proving the theorem, we introduce an algorithm that
computes the equations for any single instance $X(V)$.

\begin{prop}
There exists an algorithm $\smear$, that, on the same input as
$\parm$ plus a finite-dimensional vector space $V$, outputs generators of
the radical ideal of $X(V) \subseteq K[B \times Q(V)]$. 
\end{prop}

\begin{proof}
The algorithm $\smear(B;Q;U;f_1,\ldots,f_k;V)$ proceeds as follows:
choose identifications $U=K^m$ and $V=K^n$ and construct the
generic matrix 
\[ \psi:=\sum_{i=1}^{\dim U} \sum_{j=1}^{\dim V}
z_{ij} E_{ij} \in K[(z_{ij})_{ij}] \otimes \Hom_K(V,U),\] 
where the $E_{ij}$
form the standard basis of $\Hom_K(V,U)$. Then compute $Q(\psi) y \in
K[(z_{ij})_{ij}] \otimes Q(U)$, where $y=(y_1,\ldots,y_{\dim
Q(V)})^T$ represents a point of $Q(V) \cong K^{\dim Q(V)}$ with
variables as coordinates, substitute $Q(\psi)y$ into
the $f_i$, and expand as a polynomial in the $z_{ij}$ with coefficients
in $K[B][y_1,\ldots,y_{\dim Q(V)}] \cong K[B \times Q(V)]$.  Finally, return
generators of the radical of the ideal generated by all these
coefficients. The correctness of this algorithm follows from the
fact that those coefficients span the same space as the images
of the $f_i$ under pull-back along the linear maps $Q(\phi):Q(V)
\to Q(U)$ for all $\phi \in \Hom(V,U)$. 
\end{proof}

\subsection{The parameterisation algorithm}

The algorithm $\parm$ is recursive and proceeds as follows;
the algorithmic part is written in normal font, text that will be used
in the analysis in italic. The proofs of termination and correctness
are below.

\begin{enumerate}
\item If $Q=0$, then compute the variety $A \subseteq B$ defined
by $f_1,\ldots,f_k$ via a radical ideal computation, return $(A,0, A \hookrightarrow B)$, and exit.
\label{it:Exit1}

\item Decompose $Q=Q' \oplus R$ where $R$ is an irreducible
subfunctor of the top-degree part of $Q$.

{\em 
This corresponds to choosing a partition in the tuple representing $Q$.
Let $x_1,\ldots,x_n$ be a basis of $R(U)^*$.
We regard elements in $K[B \times Q(U)] \cong K[B \times Q'(U)]
\otimes K[R(U)] \cong K[B \times Q'(U)][x_1,\ldots,x_n]$ as polynomials
in $x_1,\ldots,x_n$ with coefficients in $K[B \times Q'(U)]$.}

\item Compute 
\[ (f_1',\ldots,f'_r):=\smear(B;Q;U;f_1,\ldots,f_k;U) \] 
and a Gr\"obner basis $(g_1,\ldots,g_l)$ of the eliminiation ideal in $K[B
\times Q'(U)]$ obtained by eliminating all $x_i$ from $f_1',\ldots,f'_r$.

{\em 
The elements $f_1',\ldots,f'_r$ generate the radical ideal of
$X(U)$ in $B \times (Q'(U) \oplus R(U))$. 
Let $X'(V)$ be the closure of the image of $X(V)$ under projection
$B \times Q(V) \to B \times Q'(V)$, so $X'$ is a closed subset
of $B \times Q'$. 
Then $g_1,\ldots,g_l$ generate the
radical ideal of $X'(U)$ by construction; and, as we will see
below in the proof of correctness, we have $X'=\cV_{B \times
Q'}(g_1,\ldots,g_l)$.}

\item In each $f_i$, replace each coefficient by its normal form modulo
$g_1,\ldots,g_l$.

{\em This has the effect that all coefficients that vanish identically
on $X'(U)$ are set to zero.}

\item If all $f_i$ are now zero, then compute
\[ (A';P';\alpha'):= \parm(B;Q';U;g_1,\ldots,g_l), \] 
output $(A',P' \oplus R, \alpha' \times \id_R)$, and exit.
\label{it:Exit2}

\item Pick the minimal $i$ for which $f_i$ is nonzero and let $x_j$ be
a variable that appears in some monomial in $f_i$ with a nonzero
coefficient. 

{\em There cannot be degree-$0$ elements among the $f_i$, because these
would lie in the elimination ideal and hence have been reduced to zero.}

\item Compute the partial derivative $h:=\partial f_i / \partial x_j
\in K[B \times Q'(U)][x_1,\ldots,x_n]$.

{\em By construction, this $h$ is nonzero and its coefficients, which
are a subset of the coefficients of $f_i$ up to some positive integer
scalars, do not lie in the ideal generated by $g_1,\ldots,g_l$.}

\item Compute 
\[ 
(A';P';\alpha'):=\parm(B;Q;U;h,f_1,\ldots,f_k).
\]
\label{it:Call4}
\vspace{-3ex}

\item Compute $\tilde{Q}:=(\Sh_{U} Q)/R$ via \cite[Exercise
6.11]{Fulton91}, determine the pure part $Q''$ of $\tilde{Q}$, and
compute $B'':=(B \times Q(U))[1/h]$.

{\em So we have $(B \times \Sh_U Q)[1/h] = B'' \times (Q'' \oplus R)$.}

\item Compute 
\[ (f_1'',\ldots,f_s''):=\smear(B;Q;U;f_1,\ldots,f_k;U \oplus U) \]
and replace each $f_i''$ by its image in $K[B'' \times (Q''(U)
\oplus R(U))]$ under the map $K[B \times Q(U \oplus U)] \to K[B''
\times (Q''(U) \oplus R(U))]$ dual to the inclusion 
\begin{align*} B'' \times (Q''(U) \oplus R(U))
&\subseteq B \times Q(U) \times (Q''(U) \oplus R(U))\\
&= B \times (\Sh_U Q)(U) = B \times Q(U \oplus U). \end{align*}
{\em The elements $f_1'',\ldots,f_s''$ generate the radical ideal of $(\Sh_U X)(U)[1/h]$
in $B'' \times (Q''(U) \oplus R(U))$.  As we will see in the proof
of correctness, we have $(\Sh_U X)[1/h]=\cV_{B'' \times (Q'' \oplus
R)}(f_1'',\ldots,f_s'')$.}

\item 
Using Buchberger's algorithm to eliminate the coordinates on
$R(U)$ from $f_1'',\ldots,f_s''$, compute equations $g''_1,\ldots,g''_t$ for the projection
of the finite-dimensional affine variety $(\Sh_{U} X)(U)[1/h]$ into $B'' \times Q''(U)$.
\label{it:linelm}

{\em Recall from \cite[Lemma 25]{Draisma17} that the projection $(B'' \times (Q'' \oplus R)) \to B'' \times
Q''$ restricts to a closed embedding from $(\Sh_{U} X)[1/h]$ to a
closed subset $X''$ of the latter space. We will see below
that $X''=\cV_{B'' \times Q''}(g''_1,\ldots,g''_t)$.}

\item Compute the inverse $\iota:X'' \to (\Sh_{U}X)[1/h]$.

{\em By Lemma~\ref{lm:Unique}, this inverse is uniquely determined by
its instance $\iota_V$ with $\dim V=\deg(Q)$.}

\item Compute
\[ (A'';P'';\alpha''):=\parm(B'';Q'';U;g''_1,\ldots,g''_t),\] 
output
\[ (A' \sqcup A'', P' \oplus P'', \alpha' \sqcup (\pi \circ \iota
\circ  \alpha'')), \]
where $\pi: B \times (\Sh_{U} Q) \to B \times Q$, evaluated at $V$,
equals $\id_B \times Q(0_{U,0} \oplus \id_V)$, and exit. \label{it:Final}

{\em Here $\alpha'$ is regarded as a map $A' \times (P' \oplus P'')
\to B \times Q$ that ignores the argument from $P''$, and similarly $\alpha''$
ignores the component in $P'$.}

\label{it:Exit3}

\end{enumerate}

\subsection{Termination of $\parm$}

\begin{proof}[Proof of termination of $\parm$.]
If, on some input, $\parm$ would not terminate, then this were due
to an infinite chain of recursive calls to itself. In the recursive
calls in steps~\eqref{it:Exit2} and~\eqref{it:Exit3}, the polynomial
functors $Q'$ and $Q''$, respectively, are smaller than $Q$ in the
order from \S\ref{ssec:Order}, whereas $Q$
remains the same in the call in step~\eqref{it:Call4}. Since the order
on polynomial functors is well-founded, apart from a finite initial
segment, the infinite chain consists entirely of consecutive calls in
step~\eqref{it:Call4}. Now note that, after each such call, $X'(U)$ either
remains constant or becomes smaller. As long as it remains
constant, i.e., the list $(g_1,\ldots,g_l)$ remains constant,
the
degree in $x_1,\ldots,x_n$ of the first equation keeps dropping. Hence
after finitely many such steps, $X'(U)$ becomes strictly
smaller. It follows that $X'(U)$ becomes smaller infinitely
often, but this contradicts the Noetherianity of the
finite-dimensional variety $B \times Q'(U)$. 
\end{proof}

\subsection{Correctness of $\parm$}

\begin{proof}[Proof that $\parm$ returns the correct output.]
We now prove that the output of the algorithm is correct. This is
immediate if the algorithm exits in step~\eqref{it:Exit1}. 

If it exits in step~\eqref{it:Exit2}, then by
Lemma~\ref{lm:Projection1} below,
$X=X' \times R$, which is parameterised by $\alpha' \times \id_R$ for
a parameterisation $\alpha'$ of $X'$. Moreover, by the same lemma, $X'$
is defined by $g_1,\ldots,g_l$.

Finally, assume that the algorithm exits in step~\eqref{it:Exit3}. We need
to show that the union of the images of $\alpha'$ and $\pi \circ \iota
\circ \alpha''$ (on $\ol{K}$-points) equals $X$.

First consider $\alpha'$, which is computed in step \eqref{it:Call4}.
The closed subset $Y:=\cV_{B \times Q}(f_1,\ldots,f_k,h)$ parameterised
by $\alpha'$ is clearly contained in $X$, so $\alpha':A' \times P' \to B
\times Q$ has image on $\ol{K}$-points contained in the $\ol{K}$-points
of $X$.

Next we argue that $\cV_{B'' \times (Q'' \oplus R)}(f_1'',\ldots,f_s'')$
is precisely $(\Sh_UX)[1/h]$. First, by construction, $f_1'',\ldots,f_s''$
generate the ideal of $(\Sh_UX)[1/h](U)$ in $(B \times Q(U \oplus
U))[1/h]$; this shows the containment $\supseteq$. For the opposite
containment we argue that if $V$ is any finite-dimensional $K$-vector
space and $(b,q) \in B \times Q(U\oplus V)$ has the property that for all
$\phi \in \Hom(V,U)$ the point $(b,Q(\id_U \oplus \phi)q) \in B \times
Q(U \oplus U)$ lies in $X(U \oplus U)$, then $(b,q) \in X(U \oplus V)$. To
show that $(b,q) \in X(U \oplus V)$, since $X$ is defined by its equations
in $K[B \times Q(U)]$, it suffices to prove that for each $\psi \in \Hom(U
\oplus V,U)$ we have $(b,Q(\psi)q) \in X(U)$. Let $\phi:=\psi|_V \in
\Hom(V,U)$ be the restriction of $\psi$ to $V$. Then 
\[ \ker(\id_U
\oplus \phi: U \oplus V \to U \oplus U) = \ker(\phi) \subseteq
\ker(\psi) \] 
and
hence
the linear map $\psi$ factors
as $\psi' \circ (\id_U \oplus \phi)$ for some $\psi' \in \Hom(U \oplus
U,U)$. Hence, since $Q$ is a functor, 
\[ (b,Q(\psi)q)=(b,Q(\psi') (Q(\id_U \oplus \phi) q))
\in X(U)\] 
because $(b,Q(\id_U \oplus \phi)q) \in X(U \oplus
U)$ and $\id_B \times Q(\psi')$ maps $X(U \oplus U)$
into $X(U)$. This concludes the proof that $\cV_{B'' \times (Q''
\oplus R)}(f_1'',\ldots,f_s'')=(\Sh_UX)[1/h]$. 

Next, by \cite[Lemma 25]{Draisma17}, the projection $B'' \times (Q''
\oplus R) \to B'' \times Q''$ restricts to an isomorphism embedding
from $(\Sh_U X)[1/h]$ to a closed subset of $B'' \times Q''$, and by
Lemma~\ref{lm:Projection2} below, the equations of $X''$ can be pulled
back from the equations of $X''(U)$, so we have indeed $X''=\cV_{B'' \times Q''}(g''_1,\ldots,g''_t)$.

Finally, consider step~\ref{it:Final}. A straightforward calculation
shows that setting $\pi_V:=\id_B \times Q(0_{U,0} \oplus \id_V)$ does
indeed yield a morphism $B \times \Sh_U Q \to B \times Q$ that maps
$\Sh_U X$ into $X$. We need to
show that if $\alpha'':A'' \times P'' \to B'' \times Q''$ is a morphism
parameterising $X''$, then $\pi \circ \iota \circ \alpha''$ is a morphism
$A'' \times P'' \to B \times Q$ whose image contains all points in $X$
that are not in the subset $Y$ parameterised by $\alpha'$. So assume that
$(b,q) \in X(V) \subseteq B \times Q(V)$ does not lie in $Y(V)$. Then
there exists a linear map $\phi:V \to U$ such that $h(b,Q(\phi)q)
\neq 0$. Let $\psi \in \Hom(V ,U \oplus V)$ be the map $v \mapsto
(\phi(v),v)$. Then, since $(\id_U \oplus 0_{V,0}) \circ \psi = \phi$,
the point $q':=Q(\psi)q \in X(U \oplus V)$ satisfies
$h(b,Q(\id_U \oplus 0_{V,0})q')=h(b,Q(\phi)q) \neq 0$, i.e., $q'$ lies
in $(\Sh_U X)[1/h]$, which is the image of $\iota \circ \alpha''$.
Moreover, since $(0_{U,0} \oplus \id_V) \circ \psi=\id_V$, we have 
$\pi_V(q')=q$. Hence $q$ lies in the image of $\pi \circ \iota \circ
\alpha''$. This concludes the proof of correctness of $\parm$.
\end{proof}

\begin{lm} \label{lm:Projection1}
Let $Q$ be a pure polynomial functor and $X \subseteq B \times
Q$ a closed subset. Assume that $X$ is defined by its equations in $K[B \times
Q(U)]$, let $R$ be a subfunctor of $Q$, and set $Q':=Q/R$. Define $X'$
as the closure of the image of $X$ under the projection $\id_B \times
\pi: B \times Q \to B \times Q'$. Assume that $X(U)=(\id_B \times
\pi_U)^{-1}(X'(U))$. Then $X(V)=(\id_B \times \pi_V)^{-1}(X'(V))$ for
all $V \in \Vec$, so that $X=X' \times R$, and moreover $X'$ is defined
by its equations in $K[B \times Q'(U)]$.
\end{lm} 

\begin{proof}
Let $(b,q) \in B \times Q(V)$ satisfy $(b,\pi_V(q))
\in X'(V)$. Then for all $\phi \in \Hom(V,U)$ we have
\[ (b,\pi_U(Q(\phi)(q))) = (b,Q'(\phi)(\pi_V(q))) \in X'(U) \]
and hence, by assumption, $(b,Q(\phi)(q)) \in X(U)$. But then, since
$X$ is defined by its equations in $K[B \times Q(U)]$, we have $(b,q)
\in X(V)$. This proves the first statement. 

For the last statement, suppose that $(b,q') \in B \times
Q'(V)$ is such that $(b,Q'(\phi)(q')) \in X'(U)$ for all $\phi
\in \Hom(V,U)$. Pick any $q \in Q(V)$ with $\pi_V(q)=q'$. Then
the same computation as above shows that $(b,q) \in X(V)$, hence {\em
a fortiori} $(b,q') \in X'(V)$. Hence $X'$ is defined by its equations
in $K[B \times Q'(U)]$.
\end{proof}

\begin{lm} \label{lm:Projection2}
Let $Q$ be a pure polynomial functor, $R$ a (not necessarily irreducible)
subfunctor, $B$ an affine variety. Set $Y:=B\times Q$, $Y':=B\times (Q/R)$. 
Let $X$ be a closed subset of $Y$ such that the projection $\pi: Y \to Y'$
restricts to an isomorphism from $X$ to a closed subset $X'$ of $Y'$.
Let $U$ be a vector space such that $X$ is defined by its equations in
$K[Y(U)]$. Then $X'$ is defined by its equations in $K[Y'(U)]$.
\end{lm}

\begin{proof}
The idea is to find an inverse to $\pi$, i.e., a morphism $\psi: Y' \to Y$
such that 
\begin{enumerate}
    \item $\pi_V \circ \psi_V = \id_{Y'(V)} $
    \item $\psi|_{X'} = (\pi|_X)^{-1}$
\end{enumerate}

Once we have found this $\psi$, we are done: Let $y' \in Y'(V)$, such that for
every linear map $\varphi: V \to U,$ $Y'(\varphi)(y') \in X'(U)$, and hence, by property (2),
$\psi_U(Y'(\varphi)(y')) \in X(U)$. Since $\psi$ is a morphism, 
we get 
$$\psi_U(Y'(\varphi)(y')) = (Y(\varphi) (\psi_V(y')) \in X(U).$$
Since $X$ is defined by its equations in $K[Y(U)]$, we get that $\psi_V(y') \in X(V)$,
and hence, with property $(1)$ of $\psi$, $\pi_V(\psi_V(y')) = y' \in X'(V)$.

By the Friedlander-Suslin-Lemma (Lemma \ref{lm:FS} and proof of
Lemma \ref{lm:Unique}), it is enough to find a $\GL_d$-equivariant map 
$\tilde{\psi}: Y'(K^d) \to Y(K^d)$, where $d$ is the degree of $Q$,
that fulfills properties (1) and (2) (with $V=K^d$ and $\psi_{K^d} = \tilde{ \psi}$).

It is easy to find a possibly non-$\GL_d$-equivariant map $\hat{\psi}:
Y'(K^d) \to Y(K^d)$ with properties (1) and (2): Note that $X'(K^d)$
is a closed subvariety of the affine space $\AA^{n'}_B$, where
$n'=\dim(Q(K^d)/R(K^d))$ and $X(K^d)$ is a closed subvariety of the
affine space $\AA^n_B$, where $n=\dim(Q(K^d))$.  By basic properties of
affine spaces, the map $(\pi_{K^d}|_{X(K^d)})^{-1}$ extends to a morphism
$\hat{\psi}$ from the ambient affine space $\AA^{n'}_B$ of $X'(K^d)$ to
the ambient affine space $\AA^n_B$ of $X(K^d)$. Now $\hat{\psi}$ lives
in some finite-dimensional space equipped with a linear action of $\GL_d$
given by $(g,\eta) \mapsto Y(g) \circ \eta \circ Y'(g^{-1})$. Since
this space is an algebraic group representation of $\GL_d$ and $\GL_d$
is linearly reductive, there exists a $\GL_d$-equivariant linear map $R$
from the space onto its subspace of $\GL_d$-invariant elements---the
so-called {\em Reynolds operator}.  Then $\tilde{\psi}:=R(\hat{\psi})$
is $\GL_d$-equivariant, and we claim that it still satisfies (1) and (2). 

By the Lefschetz principle, it is sufficient to check this when 
$K=\CC$. In this case, $\tilde{\psi}$ takes the more explicit form 
$$\tilde{\psi}:= \int_{U_d(\CC)} Y(\varphi) \circ \hat{\psi} \circ
Y'(\varphi)^{-1}\,\mathrm{d}\varphi$$
where $U_d(\CC) \subseteq \GL_d(\CC)$ is the unitary group, and the integral
is over the Haar measure (see e.g. \cite{Procesi07}) with $\int_{U_d(\CC)}\,\mathrm{d}\varphi=1$
(note that the Haar integral needs to be taken over a compact group, so we have
to use $U_d(\CC)$ instead of $\GL_d(\CC)$).

Indeed, by construction, the morphism $\tilde{\psi}$ defined by
this integral is $U_d$-equivariant. To see that it is, indeed,
$\GL_d$-equivariant, one uses that $U_d(\CC)$ is Zariski-dense in
$\GL_d(\CC)$ \cite[Section 8.6.1]{Procesi07}.  This is a version of
Weyl's well-known {\em unitary trick}.

The following chain of equations now verifies that $\tilde{\psi}$ meets property (1):
\begin{align*}
    &\pi_{\CC^d} \circ \tilde{\psi} 
    = \pi_{\CC^d} \circ \int_{U_d(\CC)} Y(\varphi) \circ
    \hat{\psi} \circ Y'(\varphi)^{-1}\,\mathrm{d}\varphi
    = \int_{U_d(\CC)} \pi_{\CC^d} \circ Y(\varphi) \circ
    \hat{\psi} \circ Y'(\varphi)^{-1}\,\mathrm{d}\varphi\\
    &= \int_{U_d(\CC)}  Y'(\varphi) \circ \pi_{\CC^d} \circ
    \hat{\psi} \circ Y'(\varphi)^{-1}\,\mathrm{d}\varphi
    = \int_{U_d(\CC)}  \id_{Y'(\CC^d)}\,\mathrm{d}\varphi.
\end{align*}
For property (2), let $x \in X'(\CC^d)$. Then
\begin{align*}
    &\tilde{\psi}(x)
    =\int_{U_d(\CC)} Y(\varphi)(\hat{\psi}
    (\underbrace{Y'(\varphi)^{-1}(x)}_{\in
    X'(\CC^d)}))\,\mathrm{d}\varphi\\
    &= \int_{U_d(\CC)} Y(\varphi)((\pi_{\CC^d}|_{X(\CC^d)})^{-1}
    (Y'(\varphi)^{-1}(x)))\,\mathrm{d}\varphi\\
    &= \int_{U_d(\CC)} Y(\varphi)(Y(\varphi)^{-1}
    ((\pi_{\CC^d}|_{X(\CC^d)})^{-1}(x)))\,\mathrm{d}\varphi
    = (\pi_{\CC^d}|_{X(\CC^d)})^{-1}(x).
\end{align*}
Note that for the second to last equality we have used that $(\pi|_X)^{-1}$
is a morphism, as is easy to verify.
\end{proof}

\section{Implicitisation}\label{sec:Implicitisation}

\subsection{The result}

The goal of this section is to prove the following more
precise version of the Main Theorem.

\begin{thm} \label{thm:impl}
There exists an algorithm $\impl$ that, on input finite-dimensional
affine varieties $A,B$, pure polynomial functors $P,Q$, and a morphism
$\alpha:A \times P \to B \times Q$, computes a $U \in \Vec$ and elements
$f_1,\ldots,f_k \in K[B \times Q(U)]$ that define the image
closure of $\alpha$, i.e., such that, for all $V \in \Vec$,
we have
\begin{align*} &\overline{\alpha(A(\ol{K}) \times \ol{K} \otimes P(V))}
\\
& = \{ (b,q) \in B(\ol{K}) \times \ol{K} \otimes Q(V) \mid
\forall i\  \forall \phi \in
\Hom_K(V,U):\ f_i(b, 1 \otimes Q(\phi)q) =0 \}.
\end{align*}
\end{thm}

\subsection{The implicitisation algorithm}

The algorithm $\impl(A,P,B,Q,\alpha)$ is run on countably many parallel
processors: the one where the original call is handled, plus countably
many labelled $0,1,2,3,\ldots$ where calls to an auxiliary procedure
$\cert$ is made. The latter has the following specification.

\begin{prop} \label{prop:Cert}
There exists a procedure that, on input finite-dimensional affine
varieties $A,A',B$, pure polynomial functors $P,P',Q$, and morphisms
$\alpha:A \times P \to B \times Q$ and $\alpha': A' \times P' \to B
\times Q$, has the following behaviour: if for all $V \in \Vec$ we have
\[ \alpha'(A'(\overline{K}) \times (\overline{K}
\otimes P'(V))) \subseteq 
\overline{\alpha(A(\overline{K}) \times (\overline{K}
\otimes P(V)))}, \]
then $\cert(B;Q;A;P;\alpha;A';P';\alpha')$ terminates and
returns ``true''. Otherwise, it does not terminate.
\end{prop}

The proof of Proposition~\ref{prop:Cert} is deferred to
\S\ref{sec:Certify}. We can now present the steps for
$\impl(A,P,B,Q,\alpha)$.

\begin{enumerate}
\item Set $n:=0$.
\item While no instance of $\cert$ has returned ``true'', 
perform steps (3)--(6):
\item Set $U:=K^n$ and, by classical elimination, compute defining
equations $f_{n,1},\ldots,f_{n,k_n} \in K[B \times Q(U)]$ for the image
closure of $\alpha_U$\\ ({\em this is a finite-dimensional affine variety}).
\item Compute $(A_n,P_n,\alpha_n):=\parm(B;Q;U;f_{n,1},\ldots,f_{n,k_n})$. 
\item On the $n$-th processor, start
$\cert(B;Q;A;P;\alpha;A_n;P_n;\alpha_n)$.
\item Set $n:=n+1$.
\item if the $m$-th processor has returned ``true'', then
return $(K^{m};f_{m,1},\ldots,f_{m,k_{m}})$. 
\end{enumerate}

\subsection{Correctness and termination of $\impl$}

\begin{proof}[Proof of Theorem~\ref{thm:impl}]
By Corollary~\ref{cor:Finiteness}, there exists a value of
$n$ such that the tuple
$(K^{n};f_{n,1},\ldots,f_{n,k_n})$ computed in iteration $n$ is a correct
output for $\impl$. It then follows that $\alpha_n$, which, by virtue
of $\parm$, parameterises the closed subset of $B \times Q$ defined by
$f_{n,1},\ldots,f_{n,k_n}$, has its image contained in the closure of
the image of $\alpha$. Hence, by Proposition~\ref{prop:Cert},
the $n$-th call to $\cert$ terminates and returns ``true''. This shows that $\impl$
terminates. 

Next, when it terminates with output
$(K^m;f_{m,1},\ldots,f_{m,k_m})$, then this is because the
image of $\alpha_m$, which equals the closed subset of
$B \times Q$ defined by $f_{m,1},\ldots,f_{m,k_m}$, is
contained in the image closure of $\alpha$. Since,
conversely, the image closure of $\alpha$ is contained in
the closed subset defined by
$f_{m,1},\ldots,f_{m,k_m}$, the output is correct. 
\end{proof}

\section{Certifying inclusion of image closures}
\label{sec:Certify}

\subsection{An instructive example}

\begin{ex} \label{ex:Limit}
Let $\alpha:(S^1)^2 \to S^3$ be the morphism defined by
$\alpha(u,v)=u^3+v^3$. The image closure of $\alpha$ is the set of
symmetric three-tensors of {\em border Waring rank} at most $2$. On the other
hand, let $\beta:(S^1)^2 \to S^3$ be the morphism defined by $\beta(u,v)=6
u^2v$. Then we have
\[ \beta(u,v)=6 u^2 v=\lim_{t \to 0} 
[(t^2 v + t^{-1}u)^3 + (t^2 v - t^{-1}u)^3] = \lim_{t \to 0}
\alpha(t^2 v + t^{-1}u,t^2 v - t^{-1}u) \]
and this implies that $\im(\beta) \subseteq
\overline{\im(\alpha)}$. However, $\im(\beta) \not \subseteq
\im(\alpha)$---indeed, it is well known that the Waring rank of cubics
of the form $u^2 v$ with $u,v$ linearly independent vectors is equal to
$3$.
\end{ex}

This example shows a {\em certificate} for $\im \beta \subseteq
\overline{\im \alpha}$, namely, that $\beta$ is a limit of a
composition $\alpha \circ \gamma_t$ with 
\[ \gamma_t:(S^1)^2 \to (S^1)^2, (u,v) \mapsto (t^2 v +
t^{-1} u, t^2 v - t^{-1} u). \]
Roughly speaking, whenever $\im(\beta)$ is contained in
$\overline{\im(\alpha)}$, where $\beta,\alpha$ are morphisms into
$B \times Q$, with $B$ a finite-dimensional variety and $Q$ a pure
polynomial functor, there is a certificate of this inclusion such as
the one above---see below for the precise statement. However, we are not
aware of any {\em a priori} lower bound on the (negative) exponents of
$t$ in such a certificate. This is why, in
Proposition~\ref{prop:Cert},
the procedure $\cert$ does not terminate when no certificate exists.

\subsection{An excursion to infinite dimensions} \label{ssec:InfDim}

We collect some material on $\GL$-varieties.  The results stated here
appear in \cite{Bik21,Bik22} or can directly be derived from results
there.

Given a finite-dimensional variety $A$ and a pure polynomial functor
$P$, we construct the inverse limit $\lim_{\ot n} (A \times P(K^n))
= A \times P_\infty$, where the projections $P(K^{n+1}) \to P(K^n)$
are of the form $P(\pi)$ with $\pi$ the standard projection $K^{n+1}
\to K^n$. Rather than as a set of $K$-valued points, we will regard $A
\times P_\infty$ as a reduced, affine $K$-scheme, namely, the spectrum
of the ring $K[A] \otimes_K R$, where $R$ is the symmetric algebra of
the countable-dimensional vector space $\lim_{n \to \infty} P(K^n)^*$.
The group $\GL:=\bigcup_{n=0}^\infty \GL_n(K)$ acts on $A \times
P_\infty$ by means of automorphisms and $A \times P_\infty$ is a {\em
$\GL$-variety} in the sense of \cite{Bik21}. More generally, if $X$ is a closed
subset of a polynomial functor, then the inverse limit $X_\infty$ of all
$X(K^n)$ is a $\GL$-variety. The association $X \mapsto X_\infty$ is an
equivalence of categories with the category of affine $\GL$-varieties,
which sends a morphism $\alpha:X \to Y$ to a $\GL$-equivariant morphism
$\alpha_\infty:X_\infty \to Y_\infty$ of affine schemes over $K$.

Let $\alpha:A \times P \to B \times Q$ and 
$\alpha':A' \times P' \to B \times Q$
be morphisms as in Proposition~\ref{prop:Cert}. Then the following two statements are equivalent:
\begin{enumerate}
\item $\overline{\im(\alpha_\infty)} \supseteq
\im(\alpha'_\infty)$ and 
\item $\overline{\im(\alpha_V)} \supseteq \im(\alpha'_V)$ for all $V
\in \Vec$. 
\end{enumerate}
It is (2) which we want to certify in
$\cert(B;Q;A;P;\alpha;A',P',\alpha')$. From now on, we assume that $A,A'$
are irreducible---in the procedure $\cert$ we will reduce to this case.

Now \cite{Bik22} contains the following useful criterion for (1).
Let $a'$ be the generic point of $A'$ and let $p' \in P'_\infty(K)$ be a
point whose $\GL$-orbit is dense in $P_{\infty}$ (such points exist, see
\cite{Bik21}). Let $x:=\alpha'_\infty(a',p')$; this is an $\Omega$-point of
$B \times Q_\infty=:X_\infty$, where $\Omega=K(A')$, and the $\GL$-orbit of
$x$ is dense in $\im(\alpha'_\infty)$. Write $Y_\infty:=A \times P_\infty$.

\begin{thm} \label{thm:Limit}
We have $\im(\alpha'_\infty) \subseteq \overline{\im(\alpha_\infty)}$
if and only if there exists a finite-dimensional field extension $\widetilde{\Omega}$ of $\Omega$
and a bounded $\widetilde{\Omega}((t))$-point $y(t)$ of $Y_\infty$ such that $\lim_{t \to
0} \alpha_\infty(y(t))=x$.
\end{thm}

This follows from \cite[Theorem 5.8]{Bik22} and its proof. Here ``bounded''
means that the exponents of $t$ in the countably many coordinates of the
$P_\infty$-component of $y(t)$ are
uniformly bounded from below. This ensures that $\alpha_\infty(y(t))$
is a well-defined $\widetilde{\Omega}((t))$-point of $X_\infty$. The theorem says that
we can choose $y(t)$ such that $\alpha_\infty(y(t))$ is in fact an
$\widetilde{\Omega}[[t]]$-point of $X_\infty$ and that setting $t$ to zero yields $x$.

The procedure $\cert$ should certify the existence of $y(t)$. To this
end, we will narrow down the space in which to search for $y(t)$ to an
increasing chain of finite-dimensional varieties. First, we will show
that $y(t)$ needs not use more of $P_\infty$ than can be obtained by
applying morphisms to $p'$. To do so (see Proposition~\ref{prop:Easier}
below), we now introduce {\em systems of variables} in Schur
functors.

\subsection{Systems of variables in Schur functors}
\label{ssec:Systems}

Fix a field extension $\Omega$ of $K$.  For every nonempty partition
$\lambda$, $S_{\lambda,\infty}$ is an affine scheme over $K$ whose
$\Omega$-valued points form an $\Omega$-vector space of uncountable
dimension. Let $V_\lambda \subseteq S_{\lambda,\infty}(\Omega)$
be the set of all points $s$ for which there exist $k$, partitions
$\mu_1,\ldots,\mu_k$ with $0<|\mu_i|<|\lambda|$, an $\Omega$-valued point
$\alpha$ of $\Map(S_{\mu_1} \oplus \cdots \oplus S_{\mu_k},S_\lambda)$,
and an $\Omega$-valued point $p$ of $S_{\mu_1,\infty} \times \cdots
\times S_{\mu_k,\infty}$ such that $\alpha_\infty(p)=s$.  (Recall the
definition of $\Map(P',P)$ from Remark~\ref{re:Map}.)

\begin{ex}
If $\lambda=(2)$, then $S_{\lambda,\infty}(\Omega)$ is the space of
infinite-by-infinite symmetric matrices with entries in $\Omega$, and
$V_\lambda$ is the subspace of matrices of finite rank.
\end{ex}

Now $V_\lambda$ is a proper $\Omega$-vector subspace of
$S_{\lambda,\infty}(\Omega)$, and we choose any $\Omega$-basis
$(\xi_{\lambda,i})_{i \in I_\lambda}$ of a vector space complement to
$V_\lambda$ in $S_{\lambda,\infty}(\Omega)$, where $I_\lambda$ is an
uncountable index set. We call the $\xi_{\lambda,i}$ {\em variables}. We
choose these variables for every $\lambda$ and write $\xi$ for the
resulting uncountable tuple; $\xi$ is what we call a {\em system of
variables} (over $\Omega$) for all Schur functors. If $f$ is an $\Omega$-valued point
of $\Map(S_{\mu_1} \oplus \cdots \oplus S_{\mu_k},S_{\lambda})$ and we
fix indices $i_1 \in I_{\mu_1},\ldots,i_k \in I_{\mu_k}$, then we will
write $f(\xi)$ for $f(\xi_{\mu_1,i_1},\ldots,\xi_{\mu_k,i_k})$. This
is slight abuse of notation, since it is not apparent from the formula
$f(\xi)$ {\em which} indices were chosen, but the notation is compatible
with the notation $f(x)$ for a polynomial that uses finitely many of an
uncountable set of variables $x$.

\begin{re}
We need to fix the extension $\Omega$ first and then choose the system
of variables. It is not true that a system of variables chosen over
$K$ is also a system of variables over field extensions of $K$, as $S_{\lambda,\infty}(K)\otimes_{K} \Omega = S_{\lambda,\infty}(\Omega)$ only when $\Omega$ is finite-dimensional over $K$. 
\end{re}

The following proposition follows readily from the material in \cite{Bik21}. 

\begin{prop} \label{prop:DenseOrbit}
Let $S=S_{\mu_1} \oplus \cdots \oplus S_{\mu_k}$ be a pure polynomial
functor, and let $s=(s_1,\ldots,s_k) \in S_\infty(\Omega)$. Then the
$s_i$ can be chosen as part of a system of variables if and only if the
$\GL$-orbit of $s$ is dense in $S_\infty(\Omega)$.
\end{prop}

The following theorem expresses that the variables in a system are
independent and generate all vectors in all Schur functors.

\begin{thm} \label{thm:Variables}
Fix a field extension $\Omega$ of $K$ and a system of
variables $\xi$ over $\Omega$ for all Schur functors. Then for every nonempty partition $\lambda$
and any $p \in S_{\lambda,\infty}(\Omega)$,
there exist partitions $\lambda_1,\ldots,\lambda_k$ and an
$\Omega$-valued point $f$ of $\Map(S_{\lambda_1} \oplus \cdots
\oplus S_{\lambda_k}, S_\lambda)$ and variables $\xi_{\lambda_j,i_j}$
for $j=1,\ldots,k$ and $i_j \in I_{\lambda_j}$ such that
$p=f_\infty(\xi_{\lambda_1,i_1},\ldots,\xi_{\lambda_k,i_k})$. Moreover, if $f$ really
depends on all $\xi_{\lambda_j,i_j}$ in the sense that replacing one of
them by zero changes the outcome, then, up to permutations of $\{1,\ldots,k\}$,
the partitions $\lambda_j$, the variables $\xi_{\lambda_j,i_j}$, and $f$ 
are unique.
\end{thm}

\begin{proof}
The existence of $f$ follows by induction on $|\mu|$: we may 
write $p$ 
\[ p=c_1 \xi_{\lambda,i_1} + c_2 \xi_{\lambda,i_2} + \cdots + c_l
\xi_{\lambda,i_l} + \tilde{p} \]
with $\tilde{p} \in V_\lambda$ and (unique) $i_1,\ldots,i_l \in I_\lambda$
and $c_1,\ldots,c_l \in \Omega$. Now $\tilde{p}=\alpha_\infty(q)$,
for a suitable $\alpha$, where $q$ is an $\Omega$-valued point
of $S_{\mu_1,\infty} \times \cdots \times S_{\mu_k,\infty}$
with $0<|\mu_i|<|\lambda|$ for all $i$.  By the induction hypothesis, the
components $q_i,i=1,\ldots,k$ of $q$ are of the form $f_{i,\infty}(\xi)$,
and then $p$ equals
\[ \alpha_\infty(f_{1,\infty}(\xi),\ldots,f_{k,\infty}(\xi))
+ \sum_{i=1}^l c_i \xi_{\lambda,i}, \]
so that $f:=\alpha(f_1,\ldots,f_k) + \sum_{i=1}^l c_i \id_{S_\lambda}$
does the trick for the obvious choice of variables. 

For uniqueness, it suffices to show that if $f$ is a nonzero
$\Omega$-valued point of $\Map(S_{\mu_1} \oplus \ldots \oplus
S_{\mu_k},S_\lambda)$ and $\xi_{\mu_1,i_1},\ldots,\xi_{\mu_k,i_k}$
are distinct variables, then $f_\infty(\xi) \neq 0$. Indeed,
by Proposition~\ref{prop:DenseOrbit}, the $\GL$-orbit of
$(\xi_{\mu_1,1},\ldots,\xi_{\mu_k,i_k})$ is dense, and $f_\infty$
is $\GL$-equivariant, so that $f_\infty(\xi)=0$ implies that $f=0$,
a contradiction.
\end{proof}

\begin{re}
Our notion of systems of variables is inspired by \cite{Erman18}.  Indeed,
If we restrict a system of variables in Schur functors to the symmetric
powers $S^d$ with $d=0,1,\ldots$, then the proof of \cite[Theorem
1.1]{Erman18} shows that these are algebraically independent (over
$\Omega$) generators of the ring $\bigoplus_d S^d_\infty(\Omega)$,
and that this ring is therefore abstractly isomorphic to a polynomial
ring. This fact can also be deduced from Theorem~\ref{thm:Variables}
above.
\end{re}

\subsection{Narrowing down the search for $y(t)$}

In this section, we retain the notation from \S\ref{ssec:InfDim}. The
following diagram represents the situation:
\[ 
\xymatrix{
\ & (a',p') \in A' \times P'_\infty \ar[d]^{\alpha'_\infty} \\
y(t)=(a(t),p(t)) \in (A \times P_\infty)(\widetilde{\Omega}((t)))
\ar[r]^{\quad \quad \alpha_\infty} & x=(b,q) \in (B \times Q_\infty)(\Omega) 
}
\]
where $\Omega=K(A')$, $a'$ is the generic point of $A'$, $p'$ is a
$K$-point of $P_\infty$ with dense $\GL$-orbit, and we want to
certify the existence of $y(t)$, defined over a finite extension
$\widetilde{\Omega}$ of $\Omega$, such that $\lim_{t \to 0}
\alpha_\infty(y(t))=x$. 

\begin{prop} \label{prop:Easier}
If a $y(t)$ as in Theorem~\ref{thm:Limit} exists, then it can be chosen of
the form $(a(t),\gamma(t)_\infty(p'))$ with $a(t) \in A(\widetilde{\Omega}((t)))$ and
$\gamma(t)$ a $\widetilde{\Omega}[t,t^{-1}]$-valued point of the
finite-dimensional affine space $\Map(P',P)$. 
\end{prop}

\begin{proof}
Write $y(t)=(a(t),p(t))$. First, terms in $p(t)$ of sufficiently high
degree in $t$ do not contribute to $\lim_{t \to 0} \alpha_\infty(a(t),p(t))$,
so we may truncate $p(t)$ and assume that it is a finite sum
$\sum_{d=m_1}^{m_2} t^d p_d$ where each $p_d \in P_\infty(\widetilde{\Omega})$.

Now write $P'=S_{\lambda_1} \oplus \cdots \oplus S_{\lambda_k}$,
where the $\lambda_i$ are partitions. Accordingly, decompose
$p'=(p'_1,\ldots,p'_k)$ with $p'_i \in S_{\lambda_i,\infty}$. Over the
field extension $\widetilde{\Omega}$ of $K$, choose a system of variables $\xi$
in such a manner that $p'_1,\ldots,p'_k$ are among these variables;
this can be done by Proposition~\ref{prop:DenseOrbit} because $\GL \cdot p'$ is dense in $P'_\infty$. Then,
by Theorem~\ref{thm:Variables}, we have $p_d=f_{d,\infty}(\xi)$ for
all $d$, where $f_d$ is an (essentially unique) morphism into $P$ with
coefficients in $\widetilde{\Omega}$.

Recall from Remark~\ref{re:Split} that $\alpha$ splits as $\alpha^{(0)}:A
\to B$ and $\alpha^{(1)}:A \times P \to Q$, and similarly for
$\alpha'$. The limit $\lim_{t \to 0} \alpha^{(1)}_\infty(a(t),p(t))
$ equals 
\[ g_\infty(p_{m_1},\ldots,p_{m_2}) = g_\infty(f_{m_1,\infty}(\xi),\ldots,
f_{m_2,\infty}(\xi)) \]
for some $\widetilde{\Omega}$-point $g$ of $\Map(P^{m_2-m_1+1},Q)$. On the other
hand, by the choice of $y(t)$, it equals $(\alpha')^{(1)}_\infty(p')$. In
the latter expression, only the variables $p'_1,\ldots,p'_k$
appear. By the uniquess statement in
Theorem~\ref{thm:Variables}, the same must apply to
$g_\infty(f_{m_1,\infty}(\xi),\ldots,f_{m_2,\infty}(\xi))$.

Therefore, replacing
each $p_d$ by $\tilde{p}_d:=f_{d,\infty}(p',0)$, where all variables not among
the variables $p_1',\ldots,p_k'$ are set to zero, yields a $\tilde{y}(t)$
with the same property as $y(t)$: $\lim_{t \to 0}
\alpha_\infty(\tilde{y}(t))=x$.
Now $\gamma(t):=\sum_d t^d f_d(\cdot,0)$ is the desired
$\widetilde{\Omega}[t,t^{-1}]$-valued point of $\Map(P',P)$. 
\end{proof}

Recall from Remark~\ref{re:Split} that $(\alpha')^{(1)}$ can be regarded
an $\Omega$-point of $\Map(P',Q)$. Similarly, $\alpha^{(1)}(a(t),\cdot)$
can be regarded an $\widetilde{\Omega}((t))$-point of $\Map(P,Q)$.

\begin{lm} \label{lm:Explicit}
A point $(a(t),\gamma(t)_\infty(p'))$ as in
Proposition~\ref{prop:Easier} satisfies the property $\lim_{t \to 0}
\alpha_\infty(a(t),\gamma(t)_\infty(p'))=\alpha'(a',p')=:(b,q) \in (B \times
Q_\infty)(\Omega)$ if and only if, first,
$ \lim_{t \to 0} \alpha^{(0)}(a(t))=b$ and, second, the $\widetilde{\Omega}((t))$-point
$\alpha^{(1)}(a(t),\cdot) \circ \gamma(t)$ of $\Map(P',Q)$
satisfies
\[ \lim_{t \to 0} \alpha^{(1)}(a(t),\cdot) \circ \gamma(t) =
(\alpha')^{(1)};
\]
an equality of $\widetilde{\Omega}$-points in $\Map(P',Q)$. 
\end{lm}

\begin{proof}
The statement ``if'' is immediate, by substituting $p'$; and the statement
``only if'' follows from the fact that the $\GL$-orbit of $p'$ is dense
in $P'_\infty$.
\end{proof}

\subsection{Greenberg's approximation theorem}

We have almost arrived at a countable chain of finite-dimensional
varieties in which we can look for $y(t)$. The only problem is that
the point $a(t) \in A(\widetilde{\Omega}((t)))$ does not yet have a finite
representation. For concreteness, assume that $A$ is given by a prime
ideal $I=(f_1,\ldots,f_r)$ in $K[x_1,\ldots,x_m]$, $B$ is embedded in
$K^n$, and $\alpha^{(0)}:A \to B$ is the restriction of some polynomial
map $\alpha^{(0)}: K^m \to K^n$.

Then $a(t)$ is an $m$-tuple in $\widetilde{\Omega}((t))^m$, and
together with $\gamma(t)$ it is required to satisfy the 
following properties from Lemma~\ref{lm:Explicit}:
\begin{enumerate}
\item[(i)] $f_i(a(t))=0$ for $i=1,\ldots,r$;
\item[(ii)] $\lim_{t \to 0} \alpha^{(0)}(a(t))=b$;
\item[(iii)] and $\lim_{t \to 0} \alpha^{(1)}(a(t),\cdot) \circ
\gamma(t) = (\alpha')^{(1)}$. 
\end{enumerate}
Suppose that we fix a lower bound $-d_1$, with $d_1 \in
\ZZ_{\geq 0}$,
on the exponents of $t$ appearing in $a(t)$ or in $\gamma(t)$. From
the data of $\alpha$ and $d_1$, one can compute a bound $d_2 \in
\ZZ_{\geq 0}$ such that the validity of (ii) and (iii) do not depend on
the terms in $a(t)$ or $\gamma(t)$ with exponents $>d_2$.
However, (i)
does depend on all (typically infinitely many) terms of $a(t)$. Here Greenberg's
approximation theorem comes to the rescue. As this 
theorem requires formal power series rather than Laurent series,
we put $\tilde{a}(t):=t^{d_1} a(t)$. Accordingly, replace each $f_i$
by $\tilde{f}_i:=t^e f_i(t^{-d_1}x_1,\ldots,t^{-d_1}x_n)$ where $e$
is large enough such that all coefficients of $\tilde{f}_i$ for all $i$
are in $\widetilde{\Omega}[[t]]$. Note that $a(t)$ is a root of all $f_i$ if and
only if $\tilde{a}(t)$ is a root of all $\tilde{f}_i$.

\begin{thm}[Greenberg, \cite{Greenberg66}]\label{thm:Greenberg}
There exists numbers $N_0 \geq 1,c \geq 1,s \geq 0$ such that for
all $N \geq N_0$ and $\overline{a}(t) \in \widetilde{\Omega}[[t]]^n$ with
$\tilde{f}_i(\overline{a}(t)) \equiv 0 \mod t^N$ for all $i=1,\ldots,r$
there exists an $\tilde{a}(t) \in \widetilde{\Omega}[[t]]^n$ such that
$\tilde{a}(t) \equiv \overline{a}(t) \mod t^{\lceil \frac{N}{c} \rceil-s}$ and moreover
$f_i(\tilde{a}(t))=0$ for all $i$. Moreover, $N_0,c,s$ can be computed from
$\tilde{f}_1,\ldots,\tilde{f}_r$.
\end{thm}

As a matter of fact, the computability, which is crucial to our work,
is only implicit in \cite{Greenberg66}; it is made explicit in the overview paper \cite{Rond18}. 

\begin{cor} \label{cor:Bound}
There exist natural numbers $d_2,N_1$, which can be computed
from $d_1$ and $f_1,\ldots,f_r$, $\alpha$, such that the following
statements are equivalent:
\begin{enumerate}
\item a pair $a(t),\gamma(t)$ with properties (i)--(iii) exists that has
no exponents of $t$ smaller than $-d_1$;
\item a pair $a(t),\gamma(t)$ exists with all exponents
of $t$ in the interval $\{-d_1,\ldots,d_2\}$ that satisfies
(ii) and (iii), and that satisfies (i) modulo $t^{N_1}$.
\end{enumerate}
\end{cor}

\begin{proof}
The implication (1) $\Rightarrow$ (2) holds for any choice of $N_1$
if $d_2$ is chosen large enough so that the terms of
$a(t),\gamma(t)$ with degree $>d_2$ in $t$ do not affect (ii),(iii),
and do not contribute to the terms of degree $<N_1$ in $f_i(a(t))$
for any $i$. 

For the converse, first we compute $\tilde{f}_i$ and $e$ as above; they
depend on the choice of $d_1$. Then we compute $N_0,c,s$ as in Greenberg's
theorem.  Compute $N_1 \geq N_0-e$ such that terms in $a(t),\gamma(t)$
in which $t$ has exponent at least $\lceil \frac{N_1+e}{c} \rceil-s-d_1$ do not affect
properties (ii)--(iii), and then compute $d_2$ as in the first paragraph.

Given a pair $a(t),\gamma(t)$ as in (2), 
set $\bar{a}(t):=t^{d_1} a(t)$. Then, for each $i$,
\[ \tilde{f}_i(\bar{a}(t))=t^e f_i(a(t)) \equiv 0 \mod t^{N_1+e}. \]
Then, since $N_1+e \geq N_0$, Greenberg's theorem yields
$\tilde{a}(t) \in \widetilde{\Omega}[[t]]^n$ such that
$\tilde{f}_i(\tilde{a}(t))=0$ for all $i$ and such that 
\[ \tilde{a}(t) = \bar{a}(t) \mod t^{\lceil \frac{N_1+e}{c} \rceil-s}. \]
Now set $a_1(t):=t^{-d_1}\tilde{a}(t)$, so that 
$f_i(a_1(t))=0$ for all $i$---this is property (i)---and 
\[ a_1(t) \equiv a(t) \mod t^{\lceil \frac{N_1+e}{c} \rceil-s-d_1}. \]
Since the terms of $a(t)$ with exponent of degree at least
$\lceil\frac{N_1+e}{c} \rceil-s-d_1$ do not affect (ii) and (iii), the pair
$a_1(t),\gamma(t)$ also satisfy these conditions.
\end{proof}

\subsection{The procedure $\cert$}

To compute $\cert(B;Q;A;P;\alpha;A';P';\alpha')$, we proceed
as follows. For convenience, we again assume that we have
sufficiently many processors working in parallel.

\begin{enumerate}

\item If $A$ and $A'$ are not both irreducible, decompose $A$ into
irreducible components $A_i$ and $A'$ into irreducible components
$A'_j$, and assign the computation of $\cert(B;Q;A_i;P;\alpha|_{A_i
\times P};A_j',P',\alpha'_{A'_j \times P'})$ for all $i,j$ to distinct
processors. As soon as for each $j$ there exists at least one $i$ such
that the computation returns ``true'', return ``true''.

({\em So in what follows we may assume that $A,A'$ are irreducible.
They are given by prime ideals $I \subseteq
K[x_1,\ldots,x_n]$ and $J \subseteq K[y_1,\ldots,y_m]$, respectively.})

\item Let $f_1,\ldots,f_r$ be generators of $I$.

\item Compute $b:=(\alpha')^{(0)}(a')$ where $a'$ is the generic point
of $A'$. 

{\em So $a'$ is just the $m$-tuple $(y_1+J,\ldots,y_m+J) \in \Omega^m$,
where $\Omega$ is the fraction field of $K[y_1,\ldots,y_m]/J$. }

\item Compute the $\Omega$-valued point $(\alpha')^{(1)}$ of
$\Map(P',Q)$. 

\item Construct a $K$-basis $\gamma_1,\ldots,\gamma_m$ of the vector
space $\Map(P',P)$. 

\item Set $d_1:=0$, $r:=$``false''.

\item While not $r$, perform the steps (8)--(10):

\item From $\alpha$ and $d_1$, compute the natural numbers $N_1,d_2$
from Corollary~\ref{cor:Bound} and make the {\em Ansatz}
$\gamma(t)=\sum_{i=1}^m c_i(t) \gamma_m$ where $c_i(t)$ is a
linear combination of $t^{-d_1},\ldots,t^{d_2}$ with coefficients
to be determined in an extension of $\Omega$; and the {\em Ansatz}
$a(t)=(a_1(t),\ldots,a_n(t))$, where $a_i$ is also a linear combination
of $t^{-d_1},\ldots,t^{d_2}$ with coefficients to be determined.

\item The desired properties of $(a(t),\gamma(t))$ from the second item
of Corollary~\ref{cor:Bound} translate into a system of polynomial
equations for the $(m+n)\cdot(d_2+d_1+1)$ coefficients of the $c_i(t)$
and the $a_i(t)$. By a Gr\"obner basis computation, test whether a
solution exists over an algebraic closure of $\Omega$. If
so, set $r:=$``true''.

\item Set $d_1:=d_1+1$. 

\item Return ``true''.
\end{enumerate}

\begin{proof}[Proof of Proposition~\ref{prop:Cert}]
The first step is justified by the observation that the
image closure of $\alpha$ contains the image of $\alpha'$ if
and only if for each $j$, the image of $\alpha'|_{A'_j
\times P'}$ is contained in the image closure of some
$\alpha|_{A_i \times P}$. 

If the image closure of $\alpha$ contains the image of $\alpha'$,
then by Theorem~\ref{thm:Limit}, Proposition~\ref{prop:Easier},
Lemma~\ref{lm:Explicit}, and Corollary~\ref{cor:Bound}, the procedure
$\cert$ terminates and returns ``true''. Otherwise, by the
same results, the system of equations in step (9) does not
have a solution, and the procedure does not terminate. 
\end{proof}

\section{Conclusion}

The Noetherianity theorem \cite{Draisma17} and the unirationality
theorem \cite{Bik21} imply that closed subsets of polynomial functors
admit two different descriptions: an implicit description by finitely many
equations, and a description as the image (closure) of some morphism. In
this paper, we have described highly nontrivial algorithms $\parm$ and
$\impl$ that translate between these two descriptions. Interestingly,
$\impl$ requires calls to $\parm$ as well as Greenberg's approximation
theorem for a stopping criterion. We do not expect these two algorithms
to be implemented in full generality anytime soon, but we do believe that
their structure may guide us in finding equations for closed subsets of
polynomial functors such as ``cubics of q-rank $\leq 2$''.

\bibliographystyle{alpha}
\bibliography{draismajournal,draismapreprint,diffeq}

\end{document}